\documentclass[11pt,a4paper]{article}
\usepackage[nottoc]{tocbibind}
\usepackage{graphicx}
\usepackage{enumerate}
\usepackage{amsmath}
\usepackage{amsfonts}
\usepackage{psfrag}
\usepackage{color}
\usepackage{pgf,tikz}
\usepackage{mathrsfs}
\usetikzlibrary{arrows}
\usepackage{fancyhdr}
\usepackage{txfonts} 

\newtheorem{theorem}{Theorem}

\newtheorem{corollary}[theorem]{Corollary}
\newtheorem{definition}{Definition}
\newtheorem{remark}{Remark}

\newtheorem{lemma}{Lemma}

\newtheorem{claim}{Claim}

\numberwithin{equation}{section}
\allowdisplaybreaks[4]

\newenvironment{proof}{\smallskip\noindent\emph{Proof.}\hspace{1pt}}%
{\hspace{-5pt}{\nobreak\quad\nobreak\hfill\nobreak$\square$\vspace{8pt}%
\par}\smallskip\goodbreak}

\newcommand{\be}{\begin{equation}}
\newcommand{\ee}{\end{equation}}

\newcommand{\supp}{\operatorname{supp}}

\newcommand{\snr}[1]{\lvert #1\rvert}

\newcommand{\nr}[1]{\lVert #1 \rVert}

\newcommand{\sh}{\tau_{hs}}
\newcommand{\hs}{\tau_{-hs}}

\newcommand{\ep}{\sh^{+}}
\newcommand{\eme}{\hs^{-}}

\newcommand{\RN}{\mathbb{R}^{N}}
\newcommand{\SN}{\mathbb{S}^{N-1}}
\newcommand{\N}{\mathbb{N}}

\newcommand{\X}{\xi}

\newcommand{\F}{\varphi}
\newcommand{\vn}{v_{j}}
\newcommand{\dd}{\mathrm{D}}

\definecolor{ffqqqq}{rgb}{1.,0.,0.}
\definecolor{uuuuuu}{rgb}{0.26666666666666666,0.26666666666666666,0.26666666666666666}

\begin{document}



\title{Higher Integrability for Constrained Minimizers of Integral Functionals with (p,q)-Growth in low dimension}

\author{
\textsc{Cristiana De Filippis\thanks{Mathematical Institute, University of Oxford, Andrew Wiles Building, Radcliffe Observatory Quarter, Woodstock Road, Oxford, OX26GG, Oxford, United Kingdom. E-mail:
      \texttt{Cristiana.DeFilippis@maths.ox.ac.uk}}}\\
}

\date{
}

\maketitle
\thispagestyle{plain}

\vspace{-0.6cm}
\begin{abstract}
  \vspace{5pt}
We prove higher summability for the gradient of minimizers of strongly convex integral functionals of the Calculus of Variations
\begin{flalign*}
\int_{\Omega}f(x,\mathrm{D}u(x)) \ dx, \quad u\colon \Omega \subset \mathbb{R}^{n}\to \SN,
\end{flalign*}
with growth conditions of $(p,q)$-type:
\begin{flalign*}
\snr{\xi}^{p}\le f(x,\X)\le C(\snr{\xi}^{q}+1), \quad p<q,
\end{flalign*}
in low dimension. Our procedure is set in the framework of Fractional Sobolev Spaces and renders the desired regularity as the result of an approximation technique relying on estimates obtained through a careful use of difference quotients.
\end{abstract}

\vspace{0.2cm}

\setlength{\voffset}{-0in} \setlength{\textheight}{0.9\textheight}

\setcounter{page}{1} \setcounter{equation}{0}


\section*{Introduction}
We deal with the regularity properties of constrained minimizers of non-autonomous integral functionals of the Calculus of Variations
\begin{flalign}\label{prob}
\mathcal{F}(u, \Omega)=\int_{\Omega}f(x,\mathrm{D}u(x)) \ dx, \quad u\colon \Omega \subset \mathbb{R}^{n}\to \SN,
\end{flalign}
where $\Omega$ is a bounded open subset in $\mathbb{R}^{n}$ and the integrand $f$ satisfies the $(p,q)$ growth condition
\begin{flalign}\label{i1}
\snr{\xi}^{p}\le f(x,\xi)\le C\left(1+\snr{\xi}^{q}\right),
\end{flalign}
$x\in \Omega$,  $\xi \in \mathbb{R}^{N\times n}$, $2\le p<q$, and $2\le n <N $.\\
A wide literature is devoted to the study of the regularity properties of unconstrained minimizers of \eqref{prob}: both the cases $p=q$ and $p<q$ are covered with a remarkable level of generality: see, among others, \cite{b1, b12, b13, b19, b20, b21} and \cite{b43, b8, b9, b10, b14, b15, b45, b29, b30, b31} respectively.\\ 
In the framework of $(p,q)$-growth, the question of regularity is particularly delicate already in the unconstrained case, since one has to take into account not only of the vectorial nature of the problem, \cite{b11, b22, b24, b48, b35, b36, b40, b41}, but also the ratio $q/p$, which can not be too large \cite{b46, b18, b47, b29, b30, b32, b33}. When a geometric constraint comes into play, the situation becomes correspondingly much more involved and, even in the simplest cases, minimizers may be singular, \cite{b26, b27}.\\
The case of constrained minimizers is very well understood for functionals with natural growth, that is $p=q$. If the target manifold exihibits a relatively simple topology, it is possible to recover the majority of the techniques available for the unconstrained case and suitably modify them in order to deal with the complications mostly due to the curvature of the constraint. In this perspective, it is worth recalling the classical \cite{b37}, where is partial regularity proved for minimizers of functionals of the type "energy plus lower order terms" defined on maps between Riemannian manifolds. Their procedure heavily relies on the substantial energy structure of the integrand, which turns out to be fundamental for deriving a scaling inequality needed for the construction of smooth maps approximating a minimizer and satisfying the constraint condition. The regularity theory is then obtained by showing that the energy has the right decay to obtain, by Morrey's Lemma, that solutions are locally H\"older continuous. We also mention \cite{b53}, which treats a general class of functionals with nice blow-ups. The novelty in his approach mainly consists in the introduction of a construction of comparison maps involving retractions in Euclidean space. This, in the case of the Dirichlet integral, simplifies the proof in \cite{b37}. A more general situation is considered in \cite{b24, b25}, where partial regularity for minimizers of functionals with $p$-growth is proved provided the target manifold is simply $[p]-1$ connected. Gehring-Giaquinta-Modica Lemma is then exploited to show higher integrability for the gradient and, as a consequence, H\"older continuity of minimizers outside a relatively closed set of zero $(n-p)$-dimensional Hausdorff measure.\\
Recently, analogous results have been obtained in \cite{b28} for holonomic minimizers of quasiconvex functionals with $p$-growth: here the question of partial regularity is handled by showing approximate harmonicity for local minimizers. This permits to compare them to solutions of bounded, linear elliptic systems with constant (frozen) coefficients to derive estimates which show that the Excess Function has the right decay. As a result, solutions have H\"older continuous gradient away from a relatively closed set of zero Lebesgue measure.\\
To the best of our knowledge, there are no such results for constrained minimizers of functionals with $(p,q)$-growth, which are interesting nevertheless, since they are related to several models used in particle physics, \cite{b16, b39}.\\
As a first step towards the developing of a complete regularity theory for solutions of this class of problems, we consider the simplified case of certain non-autonomous, strongly convex functionals with $(p,q)$-growth in low dimension, i.e. $n<p<q<N-1$. We do not use this assumption to exploit the continuity given by Morrey's Theorem for "removing" the constraint by introducing local coordinates. This hypotheses will only have a role when dealing with the critical term due to the curvature of the sphere appearing in the Euler-Lagrange equation. This combined with the fact that minimizers are, a priori, only $W^{1,p}_{\mathrm{loc}}$-functions, constitutes the major obstacle to overcome. \\
Our main result in this setting can be summarized by the following claim.
\begin{claim}
If $u \in W^{1,p}_{\mathrm{loc}}(\Omega, \SN)$ is any local minimizer of \eqref{prob}, then $\dd u \in L^{q}_{\mathrm{loc}}(\Omega, \mathbb{R}^{N\times n})$.
\end{claim}
For notation and assumptions, we refer the reader to Section \ref{S1}.\\
Let us spend a few words on our procedure. It can be framed into two steps. \\
\textbf{Step 1.} We focus on some classes of functionals with $(p,q)$-growth, not affected by the Lavrentiev Phenomenon in the unconstrained case, \cite{b6, b7, b15}, and show how to manipulate this condition in such a way that it takes into account also the presence of the constraint. In other words, we show how to transform any sequence $(u_{j})_{j \in \N}\in W^{1,q}_{\mathrm{loc}}(\Omega, \RN)$ such that
\begin{flalign*}
u_{j} \to _{j\nearrow \infty}u \ \mathrm{in} \ W^{1,p}_{\mathrm{loc}}(\Omega, \RN) \quad \mathrm{and} \quad \mathcal{F}(u_{j}, \omega )\to_{j\nearrow \infty}\mathcal{F}(u, \omega ) \ \mathrm{for \ any \ } \omega \subset \subset \Omega,
\end{flalign*}
into a sequence $(\bar{v}_{j})_{j \in \N}\subset W^{1,q}_{\mathrm{loc}}(\Omega, \SN)$ satisfying:
\begin{flalign*}
\bar{v}_{j} \to _{j\nearrow \infty}u \ \mathrm{in} \ W^{1,p}_{\mathrm{loc}}(\Omega, \SN) \quad \mathrm{and} \quad \mathcal{F}(\bar{v}_{j}, \omega )\to_{j\nearrow \infty}\mathcal{F}(u, \omega ) \ \mathrm{for \ any \ } \omega \subset \subset \Omega.
\end{flalign*}

\textbf{Step 2.} We employ the approximating sequence obtained in Step 1 to construct a family of perturbed, constrained problems as to produce a sequence of maps $(v_{j})_{j \in \N}$ with values in the unit sphere together with uniform estimates which will permit to transfer the $q$-regularity from $(v_{j})_{j \in \N}$ to $u$. In this perspective, we would like to point out an important detail: since the competitors take values in the sphere, in general we can not expect uniqueness, so we need a way to address the approximating sequence, which a priori can converge to any local minimizer of the original problem, to the "initial" one.\\
Once defined suitable, more regular, perturbed functionals, we derive the associated Euler-Lagrange equation and exploit Morrey's embedding Theorem to get uniform estimates on $(v_{j})_{j \in \N}$ by reading H\"older continuity as fractional differentiability in Fractional Sobolev spaces. Finally, we pass to the limit in $j$ to transfer the regularity of the $v_{j}$'s to $u$.\\ 
Our technique mainly relies on the one developed in \cite{b15}, to prove the equivalence between the absence of Lavrentiev Phenomenon and the extra regularity of the minimizers for unconstrained, non-autonomous variational problems.\\


\section{Notations, Main Assumptions and Functional Setting}\label{S1}
Let $\Omega\subset\mathbb{R}^{n}$ be an open, bounded subset, $n\ge 2$, $N\ge 4$, $f\colon\Omega\times\mathbb{R}^{N\times n}\to \mathbb{R}$, Carath\'eodory function.\\
We study the regularity properties of constrained minimizers of integral functionals
\begin{flalign}\label{i0}
\mathcal{F}(u,\Omega)=\int_{\Omega}f(x,\dd u(x)) \ dx, \quad u\colon \Omega \mapsto \SN.
\end{flalign}
We assume that the integrand $f$ satisfies
\begin{flalign*}
&\mathrm{(H1)}\quad \zeta \mapsto f(x,\zeta) \in C^{1}( \mathbb{R}^{N\times n}) \ \mathrm{for \ every \ } x\in \Omega, \\
&\mathrm{(H2)} \quad \snr{\zeta }^{p}\le f(x,\zeta)\le C(\snr{\zeta}^{q}+1), \quad 1\le C<\infty,\\
&\mathrm{(H3)} \quad C^{-1}(\mu^{2}+\snr{\zeta_{1}}^{2}+\snr{\zeta_{2}}^{2})^{\frac{p-2}{2}}\snr{\zeta_{1}-\zeta_{2}}\le (f'(x,\zeta_{1})-f'(x,\zeta_{2}),\zeta_{1}-\zeta_{2}), \quad 0\le \mu \le 1, \\
&\mathrm{(H4)} \quad \snr{f'(x_{1},\zeta)-f'(x_{2},\zeta)}\le C\snr{x_{1}-x_{2}}^{\sigma}(1+\snr{\zeta}^{q-1}), \quad 0<\sigma \le 1, \\
\end{flalign*}
for any $x, x_{1}, x_{2}\in \Omega$, $\zeta, \zeta_{1}, \zeta_{2}\in \mathbb{R}^{N\times n}$. The exponents $p,q$ are such that
\begin{flalign}\label{I3}
&n<p<q<N-1, \\ 
&\frac{q}{p}<\frac{n+\alpha}{n},\label{I2}
\end{flalign}
where $\alpha =\min(\sigma, \kappa)$, and $\kappa=\left(1-\frac{n}{p}\right)$ is the H\"older exponent provided by Morrey's Theorem. Moreover, by $f'(x,\zeta)$ we mean the first derivarive of $f$ in the $\zeta$ variable, i.e.: $f'(x,\zeta)=D_{\zeta}f(x,\zeta)$.\\
We shall focus on three classes of integral functionals, which clearly can be traced back to the general form \eqref{i0}:
\begin{flalign*}
(\mathcal{A})\ \  \mathcal{F}(u, \Omega)=\int_{\Omega}&g(\mathrm{D}u(x)) \ dx, \quad (\mathcal{B})\ \  \mathcal{F}(u,\Omega)=\int_{\Omega}a(x)g(\dd u(x)) \ dx, \quad (\mathcal{C}) \ \  \mathcal{F}(u, \Omega)=\int_{\Omega}g(x,\dd u(x)) \ dx.
\end{flalign*}
The integrand in $(\mathcal{A})$ covers the general autonomous case, for class $(\mathcal{B})$ we shall assume that
\begin{flalign*}
a \in C^{0,\sigma}(\Omega), \quad \sigma \in (0,1], \quad 1\le a(x)\le C, \quad \snr{\zeta}^{p}\le g(\zeta)\le C(\snr{\zeta}^{q}+1),
\end{flalign*}
and the integrand in $(\mathcal{C})$ is such that
\begin{flalign}\label{el40}
\sum_{i=1}^{4}(c_{i}^{-1}f_{i}(x,\zeta)-c_{i})\le g(x,\zeta)\le \sum_{i=1}^{4}c_{i}(f_{i}(x,\zeta)+1), \quad 1\le c_{i}<\infty,
\end{flalign}
and the $f_{i}$'s are the densities of the model energies
\begin{flalign*}
(\mathrm{E_{1}})\quad &\int_{\Omega}f_{1}(x,\dd u(x)) \ dx =\int_{\Omega}\sum_{k=1}^{n}a_{k}(x)\snr{\dd _{k}u}^{p_{k}} \ dx,\\ 
&a_{k}\in C^{0,\sigma}(\Omega), \quad 1\le a_{k}(x)\le C, \quad 2\le p=p_{1}\le p_{2}\le \cdots \le p_{n}=q,\\
(\mathrm{E_{2}})\quad &\int_{\Omega}f_{2}(x, \dd u(x)) \ dx=\int_{\Omega}\snr{\dd u}^{p}+a(x)\snr{\dd u}^{q} \ dx, \quad a \in C^{0,\sigma}(\Omega,[0,C]),\\
(\mathrm{E_{3}})\quad &\int_{\Omega}f_{3}(x,\dd u(x)) \ dx=\int_{\Omega}\snr{\dd u}^{p(x)} \ dx, \quad p \in C^{0,\sigma}(\Omega), \quad 1<p\le p(x)\le q-\varepsilon, \quad \varepsilon>0,\\
(\mathrm{E}_{4})\quad &\int_{\Omega}f_{4}(x,\dd u) \ dx=\int_{\Omega}\snr{\dd u}^{p(x)}h(\snr{\dd u}) \ dx \\ &1<p\le p(x)\le q-q_{1}-\varepsilon, \quad p\in C^{0,\sigma}(\Omega),\quad \varepsilon >0,\\
&1\le h(\snr{\zeta})\le C(1+\snr{\zeta}^{q_{1}}), \quad h \in C(\mathbb{R}^{+},[1,\infty)),\quad \mathrm{for \ any \ constant \ }p, \ \snr{\zeta}^{p}h(\snr{\zeta}) \ \mathrm{is \  convex},
 \end{flalign*}
with $\sigma \in (0,1]$.\\
For technical reasons, we shall also assume: 
\begin{flalign*}
\mathrm{(H5)}\quad &\mathrm{There \ exists} \ K>0 \  \mathrm{such \ that \ }  \int_{B^{N}_{1/2}(0)}f(x,\dd w^{a}) \ da \le K(f(x,\dd v)+1) \ \mathrm{for \ a. \ e. \ } x \in \Omega,\\ 
&\mathrm{whenever \ }w(x,a)=\snr{v(x)-a}^{-1}(v(x)-a), \mathrm{with \ }v\in W^{1,p}_{\mathrm{loc}}(\Omega) \ \mathrm{and} \ a \in B^{N}_{1/2}(0)\subset \RN.\\
&K \mathrm{\ depends \ only \ on \ }N, p, q \mathrm{\ and \ on \ the \ structure \ of \ } f, \ \mathrm{but \ not \ on \  }x, v, \mathrm{ \ or \ } a.
\end{flalign*}
\begin{remark}\label{R7}
\emph{Some comments on assumptions (H1)-(H5) are in order. First, notice that (H3) implies that }$\zeta \mapsto f(x,\zeta)$\emph{ is convex. This observation, combined with (H2) gives the existence of a positive }$c=c(n,N,q,C)$\emph{ such that, for any} $x \in \Omega$\emph{ and} $\zeta \in \mathbb{R}^{N\times n}$,
\begin{flalign}\label{el30}
\snr{f'(x,\zeta)}\le c(\snr{\zeta}^{q-1}+1).
\end{flalign}
\emph{Moreover, we do not need to assume }$f$\emph{ to be twice differentiable with respect to }$\zeta$\emph{. As a result, the strict (or degenerate) convexity of }$f$\emph{ is not formulated as}
\begin{flalign}\label{el31}
C^{-1}\snr{\zeta}^{p-2}\snr{\eta}^{2}\le f''(x,\zeta)\eta \otimes \eta\le C\snr{\zeta}^{q-2}\snr{\eta}^{2},
\end{flalign}
\emph{for any} $\zeta, \eta \in \mathbb{R}^{N\times n}$\emph{,} $x \in \Omega$\emph{. If} $f$ is $C^{2}$\emph{ in the} $\zeta$\emph{ variable, then (H4) implies the left-hand side of \eqref{el31} with a slightly different ellipticity constant, see \cite{b15, b30}.}\\
\emph{Hypotheses (H5) pays for the generality of} $f$\emph{.} \emph{In fact, as we will see in Section 3, provided some restrictions on }$h$\emph{ in (}$\mathrm{E}_{4}$\emph{), we do not need to assume it for functionals of class} $(\mathcal{C})$.
\end{remark}
Following \cite{b28}, we briefly discuss the definitions and main properties of the function spaces in which we set our problem.
\begin{definition}\label{D2}
The Sobolev space of mappings into $\SN\subset \RN$, is
\begin{flalign}\label{N11}
W^{s,p}(\Omega,\SN)=\left\{ \ u \in W^{s,p}( \Omega,\RN)\colon \ u(x) \in \SN \ \mathrm{for \ a.\ e.} \ x \in \Omega  \ \right\},
\end{flalign}
for $s\ge 1$, $1\le p\le \infty$.\\ 
The set \eqref{N11} inherits strong and weak topologies from $W^{s,p}(\Omega,\RN)$.\\
The local space $W^{s,p}_{\mathrm{loc}}(\Omega,\SN)$ consists of maps belonging to $W^{s,p}(\omega, \SN)$ for all the sets $\omega \subset \subset \Omega$.
\end{definition}

Owing to the $(p,q)$-growth behavior of $F$, we shall adopt the following notion of local minimizer.
\begin{definition}\label{D1}
A map $u\in W^{1,1}_{\mathrm{loc}}(\Omega,\SN)$ is a constrained local minimizer of $\mathcal{F}$ if and only if
\begin{flalign}\label{D11}
x\mapsto f(x,\mathrm{D}u(x))\in L^{1}_{\mathrm{loc}}(\Omega) \quad \mathrm{and} \quad \int_{\omega}f(x,\mathrm{D}u) \ dx \le \int_{\omega}f(x,\mathrm{D}v) \ dx,
\end{flalign}
for all $v \in W^{1,1}_{u}(\omega, \SN)$ and all open $\omega \subset \subset \Omega$,
where
\begin{flalign*}
&v \in W^{1,1}_{u}(\omega,\SN)=\left(u+W^{1,1}_{0}(\omega, \RN)\right)\cap W^{1,1}(\omega,\SN).\\
\end{flalign*}
\end{definition}
We stress that (H2) and $\eqref{D11}_{1}$ imply that $\mathrm{D}u \in L^{p}_{\mathrm{loc}}(\Omega,\mathbb{R}^{N\times n})$, so, if $\omega\subset \mathbb{R}^{n}$ is any set with Lipschitz boundary compactly contained in $\Omega$,
\begin{flalign*}
\inf_{v \in W^{1,1}_{u}(\omega,\mathbb{S}^{N-1})}\mathcal{F}(u,\omega)=\inf_{v \in W^{1,p}_{u}(\omega,\mathbb{S}^{N-1})}\mathcal{F}(u,\omega).
\end{flalign*}

However, in view of the $(p,q)$-growth behavior of the integrand we note that $\mathcal{F}$ is coercive on $W^{1,p}(\omega, \mathbb{R}^{N})$ while it is continuous on $W^{1,q}(\omega,\mathbb{R}^{N})$. As a consequence we are forced to consider also
\begin{flalign}\label{N10}
\inf_{v \in W^{1,q}_{u}(\omega,\mathbb{S}^{N-1})}\mathcal{F}(v,\omega)\ge \inf_{v \in W^{1,p}_{u}(\omega,\mathbb{S}^{N-1})}\mathcal{F}(v,\omega).
\end{flalign}
When strict inequality holds in \eqref{N10}, the Lavrentiev's phenomenon occurs, see \cite{b6, b7, b15}.\\
We are going to prove that 
\begin{flalign*}
\inf_{v \in W^{1,q}_{\bar{v}_{j}}(\omega,\SN)}\tilde{\mathcal{F}}_{j}(v,\omega)\to_{j \nearrow \infty} \mathcal{F}(u,\omega),
\end{flalign*}
where $(\bar{v}_{j})_{j \in \N}$ is a certain sequence converging to $u$ with respect to the $W^{1,p}$-norm and the $\tilde{\mathcal{F}}_{j}$'s are suitably perturbations of $\mathcal{F}$. This will render the local higher integrability for $\mathrm{D}u$.

\begin{remark}\label{R4}
\emph{The left-hand side in inequality \eqref{N10} raises an additional question on whether or not any minimizer} $u$\emph{, which a priori is just a }$W^{1,p}$\emph{-function, could be trace for a $W^{1,q}$-map. The answer is affirmative, since }$u(x) \in \SN$ \emph{almost everywhere, and a simple approximation argument together with embedding theorems for trace spaces prove that }
\begin{flalign*}
W^{1,p}(\partial \omega,\mathbb{S}^{N-1}) \hookrightarrow W^{1-1/q,q}(\partial \omega,\SN), 
\end{flalign*}
\emph{provided }$q<p+1$ \emph{and} $\partial \omega$ \emph{is sufficiently regular. According to \eqref{I2}, this last condition is matched.}
\end{remark}

\begin{remark}\label{R1}
\emph{When referring to balls, or, more generally, other subsets of }$\RN$\emph{ we will stress it with the apex} "${N}$" \emph{and the constant }$\omega_{N}$\emph{ will denote the} $(N-1)$\emph{-dimensional Hausdorff measure of} $\mathbb{S}^{N-1}$\emph{. Moreover, 
 if }$A\subset \mathbb{R}^{n}$\emph{ or }$A \subset \RN$\emph{ is any measurable set, we shall equally denote with }$\snr{A}$\emph{ its }$n$\emph{-dimensional Lebesgue measure or its} $N$\emph{-dimensional Lebesgue measure, the context will dissipate any confusion.}\\
\emph{Unless differently specified, }$c$\emph{ will denote a positive constant, not necessarily the same in any two occurrences, the relevant dependencies being emphasized. }
\end{remark}

\section{Known Results}
We list some well known results related to the embedding properties of functions belonging to fractional Sobolev spaces. Our main references for this section are \cite{b14, b15, b17, b21, b30}.\\
Let $h$ be any real number and $e_{s}$ the unit vector in the $s$ direction, $s \in \{1, \cdots, n\}$. For vector-valued maps $G\colon \mathbb{R}^{n}\to \RN $ we define the finite difference operators
\begin{flalign*}
\sh G(x)=G(x+he_{s})-G(x), \quad \hs G(x)=G(x)-G(x-he_{s}),
\end{flalign*}
and the left and right shift,
\begin{flalign*}
\ep G(x)=G(x+he_{s}), \quad \eme G(x)=G(x-he_{s}).
\end{flalign*}
The first lemmas explain how to control translations.
\begin{lemma}\label{L6}
For every $t$ with $1\le t<\infty$ there exists a positive constant $C$ such that
\begin{flalign*}
\int_{B_{R}(x_{0})}\left(1+\snr{G(x)}+\snr{\sh(G(x))}\right)^{t} \ dx \le C\int_{B_{2R}(x_{0})}\left(1+\snr{G(x)}\right)^{t} \ dx
\end{flalign*}
for every $G\in L^{t}(B_{2R}(x_{0}))$, for every $h$ with $\snr{h}<R$, for every $s\in \{1, \cdots, n\}$.
\end{lemma}
\begin{lemma}\label{L3}
For every $p>1$ and $G\colon B_{R}(x_{0}) \to \RN$ we have
\begin{flalign*}
\snr{\sh((\mu^{2}&+\snr{G(x)}^{2})^{(p-2)/4}G(x))}^{2}\le c(N,p)\left(\mu^{2}+\snr{G(x)}^{2}+\snr{\sh G(x)}^{2}\right)^{\frac{p-2}{2}}\snr{\sh G(x)}^{2},
\end{flalign*}
for every $x \in B_{\rho}(x_{0})$, with $\snr{h}<R-\rho$, and every $s \in \{1, \cdots, n\}$, where $0\le \mu \le 1$ and the constant $c=c(N,p)$ is independent of $\mu$.
\end{lemma}



Then, we recall the fractional Sobolev embedding theorem localized on balls. It can be obtained from the original version, set in the whole $\mathbb{R}^{n}$, by an appropriate use of a cut-off function between the balls $B_{\rho}(x_{0})$ and $B_{R}(x_{0})$. The explicit dependence of the constants on the radii $\rho$ and $R$ is also stressed.

\begin{lemma}\label{L2}
If $G\colon \mathbb{R}^{n}\to \mathbb{R}^{N}$, $G \in L^{2}(B_{R}(x_{0}))$, $0 <R\le 1$ and for some $\rho\in (0, R)$, $d\in (0,1)$, $M>0$, $\eta\colon \mathbb{R}^{n}\to \mathbb{R}$ with $\eta\colon \mathbb{R}^{n}\to \mathbb{R}$ with $\eta \in C^{\infty}_{c}(B_{(\rho+R)/2}(x_{0}))$, $0\le \eta \le 1$ in $\mathbb{R}^{n}$, $\snr{\mathrm{D}\eta}\le 4/(R-\rho)$ in $\mathbb{R}^{n}$, $\eta \equiv 1$ on $B_{\rho}(x_{0})$,
\begin{flalign*}
\sum_{s=1}^{n}\int_{B_{R}(x_{0})}\snr{\sh G(x)}^{2}\eta^{2}(x) \ dx \le M^{2}\snr{h}^{2d}
\end{flalign*}
for every $h$ with $\snr{h}\le \frac{R-\rho}{4}$, then $G \in W^{b, 2}(B_{\rho}(x_{0}))\cap L^{\frac{2n}{n-2b}}(B_{\rho}(x_{0}))$ for every $b\in (0,d)$ and
\begin{flalign*}
\nr{G}_{L^{\frac{2n}{n-2b}}(B_{\rho}(x_{0}))}\le \frac{c}{(R-\rho)^{2b+2d+2}}\left(M+\nr{G}_{L^{2}(B_{R}(x_{0}))}\right),
\end{flalign*}
where $c\equiv c(n,N,b,d)$.
\end{lemma}


The next is an iteration type result.
\begin{lemma}\label{L4}
Let $h\colon [\rho, R_{0}]\to \mathbb{R}$ be a non-negative, bounded function and $0<\theta<1$, $0\le A$, $0<\beta$. Assume that
\begin{flalign*}
h(r)\le \frac{A}{(d-r)^{\beta}}+\theta h(d),
\end{flalign*}
for $\rho \le r<d\le R_{0}$. Then
\begin{flalign*}
h(\rho)\le \frac{cA}{(R_{0}-\rho)^{\beta}},
\end{flalign*}
where $c=c(\theta, \beta)>0$.
\end{lemma}

\section{Higher Regularity for Solutions}
\subsection{Extension Results}
We start with a Lemma stating the non occurrence of the Lavrentiev Phenomenon, at least when considering $\RN$-valued approximating sequences. 
\begin{lemma}\label{L11}\emph{\cite{b15}}
Let $u \in W^{1,p}_{\mathrm{loc}}(\Omega, \SN)$ be a local minimizer of \eqref{i0}, where $\mathcal{F}$ belongs either to class $(\mathcal{A})$, $(\mathcal{B})$ or $(\mathcal{C})$.\\ There exists a sequence $(u_{j})_{j \in \N}\subset W^{1,q}_{\mathrm{loc}}(\Omega,\RN)$ such that 
\begin{flalign*}
&u_{j}\to_{j \nearrow \infty}u \ \mathrm{strongly \ in \ }W^{1,p}_{\mathrm{loc}}(\Omega, \RN), \\ &f(\cdot, \dd u_{j}(\cdot))\to_{j \nearrow \infty}f(\cdot, \dd u(\cdot)) \ \mathrm{in} \ L^{1}_{\mathrm{loc}}(\Omega).
\end{flalign*}
\end{lemma}

Since the regularity we aim to prove is of local nature, from now on we shall work on a fixed ball $B_{R}(x_{0})\subset \subset \Omega$. Clearly all the forthcoming results will be valid on any ball compactly contained in $\Omega$.\\
In the next Lemma, we are going to show how to convert the sequence $(u_{j})_{j \in \N}$ obtained in Lemma \ref{L11} into a sequence $(\bar{v}_{j})_{j \in \N}\in W^{1,q}(B_{R}(x_{0}),\SN)$ such that $\bar{v}_{j}\to_{j \nearrow \infty}u$ strongly in $W^{1,p}(B_{R}(x_{0}), \SN)$ and $f(\cdot,\dd \bar{v}_{j}(\cdot))\to f(\cdot, \dd u(\cdot))$ in $L^{1}(B_{R}(x_{0}))$.\\
Our construction is inspired by \cite{b24, b34}.
\begin{lemma}\label{L12}
Let $B_{R}(x_{0})\subset \subset \Omega$ and $u\in W^{1,p}(B_{R}(x_{0}), \SN)$. For any $(u_{j})_{j \in \N}\subset W^{1,q}(B_{R}(x_{0}), \RN)$ such that 
\begin{itemize}
    \item[\textit{i}.]$u_{j}\to_{j \nearrow \infty} u$ strongly in $W^{1,p}(B_{R}(x_{0}), \RN)$,
    \item[\textit{ii}.]$f(\cdot,\dd u_{j}(\cdot))\to_{j \nearrow \infty}f(\cdot,\dd u(\cdot)) $ in $L^{1}(B_{R}(x_{0})),$
\end{itemize}
there exists a sequence $(\bar{v}_{j})_{j \in \N}\subset W^{1,q}(B_{R}(x_{0}),\SN)$ such that
\begin{itemize}
    \item $\bar{v}_{j}\to_{j \nearrow \infty}u \ \mathrm{strongly \ in \ }W^{1,p}(B_{R}(x_{0}), \SN)$,
    \item $f(\cdot,\dd \bar{v}_{j}(\cdot))\to_{j \nearrow \infty}f(\cdot,\dd u(\cdot)) \ \mathrm{in} \ L^{1}(B_{R}(x_{0}))$.
\end{itemize}
\end{lemma}
\begin{proof}
For the sake of clarity, we split this proof into two steps.\\
\textbf{Step 1.} - $(\bar{v}_{j})_{j \in \N}\subset W^{1,q}(B_{R}(x_{0}), \SN)$, $\dd \bar{v}_{j}\to_{j \nearrow \infty}\dd u$ strongly in $L^{p}(B_{R}(x_{0}),\mathbb{R}^{N\times n})$ and $f(\cdot,\dd\bar{ v}_{j}(\cdot))\to_{j \nearrow \infty}f(\cdot,\dd u(\cdot))$ in $L^{1}(B_{R}(x_{0}))$. \\
For $a \in B^{N}_{1/2}(0)\subset \RN$, define \begin{flalign*}w^{a}_{j}(x)=\frac{u_{j}(x)-a}{\snr{u_{j}(x)-a}}. 
\end{flalign*}
By construction, $w^{a}_{j}(x) \in \SN$ for a. e. $x \in B_{R}(x_{0})$. Moreover,
\begin{flalign*}
\dd w^{a}_{j}(x)=\snr{u_{j}(x)-a}^{-1}\dd u_{j}(x)-\snr{u_{j}(x)-a}^{-3}\left((u_{j}(x)-a)\otimes(u_{j}(x)-a)\right)\dd u_{j}(x),
\end{flalign*}
so
\begin{flalign*}
\snr{\dd w^{a}_{j}(x)}\le 2\snr{u_{j}(x)-a}^{-1}\snr{\dd u(x)},
\end{flalign*}
for a. e. $x \in B_{R}(x_{0})$, for all $a \in B^{N}_{1/2}(0)$ and
\begin{flalign}\label{101}
\int_{B^{N}_{1/2}(0)}\snr{\dd w^{a}_{j}}^{p} \ da \le C(N,p)\snr{\dd u_{j}}^{p}, \quad \int_{B^{N}_{1/2}(0)}\snr{\dd w^{a}_{j}}^{q} \ dx\le C(N,q)\snr{\dd u_{j}}^{q}.
\end{flalign}
Fubini-Tonelli's Theorem and \eqref{101} yield that
\begin{flalign*}
\int_{B^{N}_{1/2}(0)}\int_{B_{R}(x_{0})}\snr{\dd w^{a}_{j}}^{q} \ dadx\le \int_{B_{R}(x_{0})}\int_{B^{N}_{1/2}(0)}\snr{\dd w^{a}_{j}}^{q} \ dadx\le C(N,q)\int_{B_{R}(x_{0})}\snr{\dd u_{j}}^{q} \ dx<\infty,
\end{flalign*}
so the map $a \mapsto \int_{B_{R}(x_{0})}\snr{\dd w^{a}_{j}}^{q} \ dx$ is in $L^{1}(B^{N}_{1/2}(0))$, thus
\begin{flalign}\label{109}
\int_{B_{R}(x_{0})}\snr{\dd w^{a}_{j}}^{q} \ dx <\infty \quad \mathrm{for \ a. \ e. \ } a\in B^{N}_{1/2}(0).
\end{flalign}
Now consider the sets 
\begin{flalign*}
K_{j}=\left\{ x \in B_{R}(x_{0})\colon \snr{u_{j}(x)}\le \frac{3}{4} \right\}.
\end{flalign*}
Since by (\textit{i}.), $u_{j}\to_{j \nearrow \infty} u$ strongly in $L^{p}(B_{R}(x_{0}),\RN)$, then
\begin{flalign}\label{102}
\snr{K_{j}}\to_{j \nearrow \infty}0. 
\end{flalign}
By Fubini-Tonelli's Theorem and (\textit{i}.),
\begin{flalign}
&\int_{B^{N}_{1/2}(0)}\int_{K_{j}}\snr{\dd w^{a}_{j}}^{p} \ dx da=\int_{K_{j}}\int_{B^{N}_{1/2}(0)}\snr{\dd w^{a}_{j}}^{p} \ da dx\le C(N,p)\int_{K_{j}}\snr{\dd u_{j}}^{p} \ dx\to_{j \nearrow \infty}0,\label{104}
\end{flalign}
thus
\begin{flalign}\label{105}
&\int_{K_{j}}\snr{\dd w^{a}_{j}}^{p} \ dx \to_{j \nearrow \infty}0 \quad \mathrm{for \ a. \ e. \ }a \in B^{N}_{1/2}(0).
\end{flalign}
Now we need to show that for any $j \in \N$ there exists an $a_{j}\in B^{N}_{1/2}(0)$ such that $f(x,\dd w^{a_{j}}_{j})\le C(N,p,q)(f(x,\dd u_{j})+1)$ a. e. in $B_{R}(x_{0})$, where $C=N\omega_{N}^{-1}2^{N+1}K$ and $K$ is the constant given by (H5), and this is straightforward to obtain. In fact, using Markov's Inequality, (H5) and the definition of $C$, we get that
\begin{flalign*}
\left |\left \{ a \in B^{N}_{1/2}(0)\colon \ f(x,\dd w^{a}_{j})\le C(f(x, \dd u_{j})+1) \ \mathrm{for \ a. \ e. \ } x \in B_{R}(x_{0})  \right\} \right |>0.
\end{flalign*}
Summarizing, we can conclude that for any $j \in \N$, there exists an $a_{j} \in B^{N}_{1/2}(0)$ such that 
\begin{flalign}\label{106}
&\int_{K_{j}}\snr{\dd \bar{v}_{j}}^{p} \ dx \to_{j \nearrow \infty}0,\quad \int_{B_{R}(x_{0})}\snr{\dd \bar{v}_{j}}^{q} \ dx< \infty, \quad f(x,\dd \bar{v}_{j})\le C(f(x,\dd u_{j})+1) \  \mathrm{a. \ e. \ in \ }B_{R}(x_{0}).
\end{flalign}
with $C=C(N,p,q)>0$ and $\bar{v}_{j}=w^{a_{j}}_{j}$, so, in particular, $(\bar{v}_{j})_{j \in \N}\subset W^{1,q}(B_{R}(x_{0}), \SN)$.\\
Now, remembering the definition of the $K_{j}$'s, from  $\eqref{106}_{1}$ and the Dominated Convergence Theorem we obtain
\begin{flalign}\label{108}
\int_{B_{R}(x_{0})}\snr{\dd \bar{v}_{j}-\dd u}^{p} \ dx &=\int_{B_{R}(x_{0})\setminus K_{j}}\snr{\dd \bar{v}_{j}-\dd u}^{p} \ dx+\int_{K_{j}}\snr{\dd \bar{v}_{j}-\dd u}^{p} \ dx\nonumber\\
&\le 2^{p-1}\left(\int_{K_{j}}\snr{\dd \bar{v}_{j}}^{p} \ dx + \int_{K_{j}}\snr{\dd u}^{p} \ dx\right)+\int_{B_{R}(x_{0})\setminus K_{j}}\snr{\dd \bar{v}_{j}-\dd u}^{p} \ dx\to_{j \nearrow \infty}0
\end{flalign}
and, as a direct consequence of $\eqref{106}_{3}$, \eqref{108} and (\textit{ii}.), by the Dominated Convergence Theorem we obtain
\begin{flalign*}
\int_{B_{R}(x_{0})}\snr{f(x,\dd \bar{v}_{j})-f(x,\dd u)} \ dx\to_{j \nearrow \infty}0.
\end{flalign*}
\textbf{Step 2.} - $(\bar{v}_{j})_{j \in N}\to_{j \nearrow \infty} u$ strongly in $W^{1,p}(B_{R}(x_{0}), \SN)$.\\
From \eqref{108} we obtain that $\dd \bar{v}_{j}\to_{j \nearrow \infty}\dd u$ strongly in $L^{p}(B_{R}(x_{0}), \mathbb{R}^{N\times n})$, so we can conclude that $\bar{v}_{j}\to_{j \nearrow \infty}\tilde{u}=u+\tilde{a}$ in $L^{p}(B_{R}(x_{0}), \SN)$, where $\tilde{a}\in \RN$ is a constant vector.\\
Using the identity $u\cdot\dd u=0$, it is easy to see that either $\tilde{u}=u$ or $\tilde{u}=-u$, and the last one is excluded by the strong convergence of $(\dd \bar{v}_{j})_{j \in \N}$ to $\dd u$.\\
For later uses, we stress that the Trace Inequality renders
\begin{flalign}\label{100}
\nr{\bar{v}_{j}-u}_{L^{p}(\partial B_{R}(x_{0}), \SN)}\le C(n,N,p)\nr{\bar{v}_{j}-u}_{W^{1,p}(B_{R}(x_{0}), \SN)}\to_{j \nearrow \infty}0.
\end{flalign}
\end{proof}
\begin{corollary}\label{C4}
Let $B_{R}(x_{0})\subset \subset \Omega$ and $u\in W^{1,p}_{\mathrm{loc}}(\Omega, \SN)$ be a local minimizer of \eqref{i0}. Then there exists a sequence $(\bar{v}_{j})_{j \in \N}\subset W^{1,q}(B_{R}(x_{0}), \SN)$ such that
\begin{itemize}
    \item[\textit{i.}]$\bar{v}_{j}\to_{j \nearrow \infty}u$ strongly in $W^{1,p}(B_{R}(x_{0}), \SN)$,
    \item[\textit{ii.}]$f(\cdot,\dd \bar{v}_{j}(\cdot))\to_{j \nearrow \infty}f(\cdot,\dd u(\cdot))$ in $L^{1}(B_{R}(x_{0}))$.
\end{itemize}
\end{corollary}
\begin{proof}
The thesis follows directly by applying Lemmas \ref{L11} and \ref{L12} to $u$.
\end{proof}
\begin{remark}
\emph{From the construction developed in Lemma \ref{L12}, we can deduce that we do not need to assume (H5) for functionals of type }$(\mathcal{C})$\emph{, if in addition we assume that}
\begin{flalign*}
h\colon \mathrm{concave, \ non \ decreasing \ and \ there \  exists \ } M>0\colon h(ct)\le Mch(t) \ \mathrm{for \ any \ }c\ge 1,
\end{flalign*}
\emph{where }$h$\emph{ is as in (}$E_{4}$\emph{).}\\
\emph{In fact, suppose for the moment that the densities} $f_{i}$\emph{,} $i \in \{1,2,3,4\}$\emph{, satisfy (H5). If} $g$\emph{ is any member of class }$(\mathcal{C})$\emph{, then}
\begin{flalign*}
\int_{B^{N}_{1/2}(0)}g(x, \dd w^{a}(x)) \ da &\le \sum_{i=1}^{4}c_{i}\left(\int_{B_{1/2}^{N}(0)}f_{i}(x, \dd w^{a}(x)) \ da+1\right)\\
&\le \sum_{i=1}^{4}c_{i}(K(f(x, \dd w(x))+1)+1)\\
&\le 2(K+1)\tilde{c}(g(x,\dd w(x))+1),
\end{flalign*}
\emph{where the }$w^{a}$\emph{'s are as in (H5), we used \eqref{el40} and } $\tilde{c}=\max_{i \in \{1,2,3,4\}}(c_{i}^{3})+1$\emph{.} \\
\emph{Hence, we only need to show (H5) for the }$f_{i}$\emph{'s. Given their structure, it is immediate to verify (H5) for }$f_{1}$, $f_{2}$\emph{ and }$f_{3}$\emph{, so we shall focus only on }$f_{4}$\emph{.}\\
\emph{Define the following quantities:}
\begin{flalign*}
d\mu =\snr{w-a}^{-p(x)}da, \quad L(x)=\mu_{x}(B_{1/2}^{N}(0))=\int_{B^{N}_{1/2}(0)}\snr{w-a}^{-p(x)} \ da,
\end{flalign*}
\emph{and notice that}
\begin{flalign*}
\frac{2^{p-N}\omega_{N}}{N-p}\le L(x)\le \frac{2^{q-N}\omega_{n}}{N-q} \quad \mathrm{for \ all \ } x \in \Omega.
\end{flalign*}
\emph{Exploiting the properties of} $h$\emph{ and the bounds on }$L$\emph{, by Jensen's Inequality we obtain}
\begin{flalign*}
\int_{B^{N}_{1/2}(0)}\snr{\dd w}^{p(x)}h(\snr{\dd w}) \ da&\le 2^{q}\snr{\dd w}^{p(x)}\int_{B^{N}_{1/2}(0)}\snr{w-a}^{-p(x)}h(2\snr{w-a}^{-1}\snr{\dd w}) \ da\\
&=2L(x)\snr{\dd w}^{p(x)}\fint_{B^{N}_{1/2}(0)}h(2\snr{w-a}^{-1}\snr{\dd w}) \ d\mu_{x} \\
&\le\frac{2^{2q-N}\omega_{N}}{(N-q)}\snr{\dd w}^{p(x)}h\left(\snr{\dd w}\fint_{B^{N}_{1/2}(0)}2\snr{w-a}^{-1} \ d\mu_{x} \ \right)\\
&\le C(N, p,q)M\snr{\dd w}^{p(x)}h(\snr{\dd w}).
\end{flalign*}
\end{remark}
\begin{remark}
\emph{It is worth noticing that we never used assumption \eqref{I3} for this part. We are aware of the fact that when \eqref{I3} is in force, then it is possible to obtain a sequence} $(\bar{v}_{j})_{j \in \N} \subset W^{1,q}_{\mathrm{loc}}(\Omega, \SN)$ \emph{strongly converging to }$u$\emph{ in the }$W^{1,p}$\emph{-norm by directly projecting on} $\SN$ \emph{the mollified sequence} $(u_{j})_{j\in \N}=(\varphi_{j}*u)_{j\in \N}$\emph{, which is well defined by means of the uniform convergence given by Morrey's Theorem. However, because of the generality of the integrand and of its }$(p,q)$\emph{-growth, it is not very clear how to show in this case the convergence of }$f(\cdot,\dd \bar{v}_{j}(\cdot))$\emph{ to} $f(\cdot,\dd u(\cdot))$\emph{ in the} $L^{1}$\emph{-norm.}
\end{remark}
\subsection{Higher Integrability for the Gradient}
Here we use our Extension results in order to develop the announced approximation scheme.\\
For $B_{R}(x_{0})\subset \subset \Omega$ and for any $j \in \N$ we define the numbers
\begin{flalign*}
\varepsilon_{j}=\left(1+j+\nr{\mathrm{D}\bar{v}_{j}}_{L^{q}({B_{R}}(x_{0}), \mathbb{R}^{N\times n})}^{3q}\right)^{-1}\to_{j \nearrow \infty}0,
\end{flalign*}
where $(\bar{v}_{j})_{j \in \N}\subset W^{1,q}(B_{R}(x_{0}), \SN)$ is the sequence provided by Corollary \ref{C4}, and the functions
\begin{flalign*}
H^{j}(z, \zeta)=f_{j}(x,\zeta)+\frac{1}{p}\snr{z-u}^{p}=f(x,\zeta)+\frac{1}{p}\snr{z-u}^{p}+\varepsilon_{j}(1+\snr{\zeta}^{2})^{q/2},
\end{flalign*}
with $z \in \mathbb{S}^{N-1}$, $\zeta \in \mathbb{R}^{N\times n}$. The functions $H^{j}$ satisfy the following properties:
\begin{flalign}
&\varepsilon_{j}\snr{\zeta}^{q}+\snr{\zeta}^{p}+\frac{1}{p}\snr{z-u}^{p}\le H^{j}(x, z, \zeta)\le C(1+\snr{\zeta}^{q}), \label{el32}\\
&\snr{\mathrm{D}_{\zeta}H^{j}(x,z,\zeta)}\le C(1+\snr{\zeta}^{q-1}),\label{el34}\\
&C^{-1}(\mu^{2}+\snr{\zeta_{1}}^{2}+\snr{\zeta_{2}}^{2})^{\frac{p-2}{2}}\snr{\zeta_{1}-\zeta_{2}}^{2}\le \left( \left(\mathrm{D}_{\zeta}H^{j}(z,\zeta_{1})-\mathrm{D}_{\zeta}H^{j}(z, \zeta_{2})\right)\cdot (\zeta_{1}-\zeta_{2} )\right)\label{el35},\\
&\snr{\dd _{\zeta}H^{j}(x_{1},z,\zeta)-\dd _{\zeta}H^{j}(x_{2},z,\zeta)}\le C\snr{x_{1}-x_{2}}^{\sigma}(\snr{\zeta}^{q-1}+1)\label{el37},
\end{flalign}
which allow us to find solutions $\vn \in W^{1,q}_{\bar{v}_{j}}(B_{R}(x_{0}),\mathbb{S}^{N-1})$ to the Dirichlet problems
\begin{flalign}\label{dp}
\min_{v\in W^{1,q}_{\bar{v}_{j}}(B_{R}(x_{0}),\mathbb{S}^{N-1})}\int_{B_{R}(x_{0})}f(x,\mathrm{D}v)+\frac{1}{p}\snr{v-u}^{p} \ dx+\varepsilon_{j}\int_{B_{R}(x_{0})}(1+\snr{\mathrm{D}v}^{2})^{q/2} \ dx.
\end{flalign}
\begin{remark}\label{R10}
$W^{1,q}(B_{R}(x_{0}), \SN)$\emph{ is not a linear space, so, apart from some special cases, \cite{b26, b27}, it is not possible to exploit the convexity of the integrand to guarantee uniqueness of solutions. This is not a problem, since for completing our regularity proof we only need the existence of at least a solution to \eqref{dp}, which results from Direct Methods, \cite{b21}.}
\end{remark}


As a first step towards the proof of the higher integrability of $\mathrm{D}u$, we derive the Euler-Lagrange equation for the family of constrained problems $\eqref{dp}$.
\begin{lemma}[Euler-Lagrange Equation]\label{L5}
Let $\vn \in W^{1,q}_{\bar{v}_{j}} (B_{R}(x_{0}), \mathbb{S}^{N-1})$ be any solution to the constrained variational problem \eqref{dp}. Then, there holds
\begin{flalign}\label{EL}
0&=\int_{B_{R}(x_{0})}f'_{j}(x,\mathrm{D}\vn)\cdot \left(\mathrm{Id}-\vn\otimes \vn\right)\mathrm{D}\varphi-f_{j}'(x,\mathrm{D}\vn)\cdot\left((\vn\cdot \varphi)\mathrm{Id}+\left(\varphi \otimes \vn + \vn \otimes \varphi\right)\right)\mathrm{D}\vn \  dx\nonumber \\
&+\int_{B_{R}(x_{0})}\snr{\vn-u}^{p-2}(\vn-u)(\mathrm{Id}-v_{j}\otimes v_{j})\F \ dx,
\end{flalign}
for any $\varphi \in C^{\infty}_{c}(B_{R}(x_{0}),\mathbb{R}^{N})$.
\end{lemma}
\begin{proof}
Take $\varphi \in C^{\infty}_{c}(B_{R}(x_{0}), \mathbb{R}^{N})$, $s \in (-\delta, \delta)$ and define $\vn^{s}=\frac{\vn+s\varphi}{\snr{\vn+s\varphi}}$.\\ 
Since $\snr{\vn}=1$, choosing $\delta$ small enough, $\vn^{s}$ is a well defined variation on the target for $s \in (-\delta, \delta)$.\\
We compute
\begin{flalign}\label{el5}
\left. \frac{d}{ds}\vn^{s} \ \right |_{s=0}&=(\mathrm{Id}-v_{j}\otimes v_{j})\F,\\
\left. \frac{d}{ds}(\mathrm{D}\vn^{s})\right |_{s=0}&=\left(\mathrm{Id}-\vn\otimes \vn\right)\mathrm{D}\varphi-\left[(\vn\cdot \varphi) \mathrm{Id}+\left(\varphi\otimes \vn+\vn\otimes \varphi\right)\right]\mathrm{D}\vn.\label{el3}
\end{flalign}

Using the fact that $\vn$ is a minimizer, together with \eqref{el5} and \eqref{el3} we obtain
\begin{flalign*}
0&=\left.\left(\int_{B_{R}(x_{0})}f_{j}(x,\mathrm{D}\vn^{s})+\frac{1}{p}\snr{\vn^{s}-u}^{p} \ dx\right) \ \right |_{s=0}\\
&=\left.\left(\int_{B_{R}(x_{0})}f'(x,\mathrm{D}\vn^{s})\cdot \frac{d}{ds}(\mathrm{D}\vn^{s})+\snr{\vn^{s}-u}^{p-2}(\vn^{s}-u)\cdot \frac{d}{ds}\vn^{s} \ dx\right) \ \right |_{s=0}\\
&=\int_{B_{R}(x_{0})}f'_{j}(x,\mathrm{D}\vn)\cdot \left(\mathrm{Id}-\vn\otimes \vn\right)\mathrm{D}\varphi-f'(x,\mathrm{D}\vn)\cdot\left[(\vn\cdot \varphi) \mathrm{Id}+\left(\varphi\otimes \vn+\vn\otimes \varphi\right)\right]\mathrm{D}\vn \  dx\\
&+\int_{B_{R}(x_{0})}\snr{\vn-u}^{p-2}(\vn-u)\cdot (\mathrm{Id}-v_{j}\otimes v_{j})\F \ dx,
\end{flalign*}
which is \eqref{EL}.
\end{proof}
\begin{remark}\label{R3}
\emph{In spite of the severe complications induced by the constraint, formulation \eqref{EL} permits to reflect the constraint condition on the structure of the integrand, so the test function }$\F$\emph{ is totally free.} \\
\emph{Furthermore, by density, \eqref{EL} holds for any} $\F \in W^{1,q}_{0}(B_{R}(x_{0}), \RN)$\emph{.}
\end{remark}
Now we can state our main theorem, which exploits the procedure developed in \cite{b15, b54}.
\begin{theorem}\label{T1}
Let $\mathcal{F}$ belong to classes ($\mathcal{A}$), ($\mathcal{B}$) or ($\mathcal{C}$), with $f$ verifying assumptions (H0)-(H5), and $p,q$ be such that \eqref{I3}-\eqref{I2} hold.\\
If $u \in W^{1,p}_{\mathrm{loc}}(\Omega,\mathbb{S}^{N-1})$ is any local minimizer of \eqref{prob}, then 
\begin{flalign*}
u \in W^{1,t}_{\mathrm{loc}}(\Omega,\mathbb{S}^{N-1}) \ \mathrm{for\ all} \ t \in \left(1,\frac{np}{n-\alpha}\right).
\end{flalign*}
Moreover, for any $0<\rho<R$ there exist three constants 
\begin{flalign*}
c=c(n,N,p,q, t, \alpha, C)<\infty, \quad \tilde{\beta}=\tilde{\beta}(n,p,q,t,\alpha)>0, \quad \tilde{\alpha}=\tilde{\alpha}(n,p,q,t,\alpha)>0,
\end{flalign*}
such that
\begin{flalign*}
\int_{B_{\rho}(x_{0})}\snr{\mathrm{D}u}^{t} \ dx\le \frac{c}{(R-\rho)^{\tilde{\alpha}}}\left(1+\int_{B_{R}(x_{0})}f(x,\mathrm{D}u)\ dx\right)^{ \ \tilde{ \beta}}.
\end{flalign*}
\end{theorem}
\begin{proof}
Notice that, taking $j \in \N$ sufficiently large, by \eqref{el32} and minimality we have that
\begin{flalign}\label{el26}
\int_{B_{R}(x_{0})}\snr{\dd v_{j}}^{p} \ dx &\le \int_{B_{R}(x_{0})}f(x, \dd v_{j})+\frac{1}{p}\snr{v_{j}-u}^{p} \ dx +\varepsilon_{j}\int_{B_{R}(x_{0})}(1+\snr{\dd v_{j}}^{2})^{q/2} \ dx\nonumber \\
&\le \int_{B_{R}(x_{0})}f(x, \dd \bar{v}_{j})+\frac{1}{p}\snr{\bar{v}_{j}-u}^{p} \ dx +\varepsilon_{j}\int_{B_{R}(x_{0})}(1+\snr{\dd \bar{v}_{j}}^{2})^{q/2} \ dx\nonumber \\
&\le\int_{B_{R}(x_{0})}f(x, \dd u) \ dx+1,
\end{flalign}
where $(\bar{v}_{j})_{j \in \N}\subset W^{1,q}(B_{R}(x_{0}), \SN)$ is the sequence given by Corollary \ref{C4}.\\ 
So, by virtue of Morrey's Theorem, we can uniformly bound the oscillation of the $v_{j}$'s as
\begin{flalign}\label{el12}
\snr{v_{j}(x)-v_{j}(y)}&\le c(n,N, p)\snr{x-y}^{\alpha}\nr{\dd v_{j}}_{W^{1,p}(B_{R}(x_{0}),\RN)}\nonumber \\
&\le c(n,N, p)\snr{x-y}^{\alpha}\left(\int_{B_{R}(x_{0})}f(x,\dd u)+1 \ dx\right)^{1/p}=C(n,N,p)m^{1/p}\snr{x-y}^{\alpha}, 
\end{flalign}
where $m=\int_{B_{R}(x_{0})}f(x,\dd u)\ dx$+1.\\
We take $0<\rho\le r<d<R\le 1$ and $\eta\colon B_{R}(x_{0})\to [0,\infty) $ such that
\begin{flalign}\label{90}
\supp(\eta)\subset \subset B_{\frac{d+r}{2}}(x_{0}),\quad 0\le \eta\le 1, \quad \eta \equiv 1 \ \mathrm{in} \ B_{r}(x_{0}), \quad \snr{\mathrm{D}\eta}\le \frac{4}{(d-r)}.
\end{flalign}
Moreover, in the following, $s \in \{1,\cdots, n\}$ and $h \in \mathbb{R}$ satisfies
\begin{flalign}\label{h}
0<\snr{h}\le \frac{d-r}{4}.
\end{flalign}
Testing \eqref{EL} against $\F=\hs(\eta^{2}\sh(\vn))$, which is admissible by Remark \ref{R3} and because 
\begin{flalign*}\vn \in W^{1,q}(B_{R}(x_{0}), \mathbb{S}^{N-1}),
\end{flalign*}
and taking into account of the fact that $\sh$ and $\mathrm{D}$ commute, we obtain
\begin{flalign*}
0 &=\int_{B_{R}(x_{0})}f'_{j}(x,\mathrm{D}\vn)\cdot (\mathrm{Id}-\vn \otimes \vn)\hs\left(\mathrm{D}(\eta^{2}(\sh\vn))\right) \ dx\\
&-\int_{B_{R}(x_{0})}f'_{j}(x,\mathrm{D}\vn)\cdot\left(\vn\cdot \hs(\eta^{2}(\sh\vn))\right)\mathrm{D}\vn \ dx\\
&-\int_{B_{R}(x_{0})}f'_{j}(x,\mathrm{D} \vn)\cdot\left(\vn \otimes \hs(\eta^{2}(\sh\vn))+\hs(\eta^{2}(\sh\vn))\otimes \vn\right)\mathrm{D}\vn \ dx\\
&+\int_{B_{R}(x_{0})}\snr{\vn -u}^{p-2}(\vn-u)\cdot(Id-v_{j}\otimes v_{j})\hs(\eta^{2}\sh v_{j}) \ dx=(\mathrm{I})+(\mathrm{II})+(\mathrm{III})+(\mathrm{IV}).
\end{flalign*}

We use integration by parts formula and multiplicative rule for finite difference operators, [15], as to obtain a suitable expression of terms (I)-(IV), and then estimate them.\\

\textbf{Term (I).}\\
\begin{flalign*}
(\mathrm{I})&=-\int_{B_{R}(x_{0})}\sh(f'_{j}(x,\mathrm{D}\vn))(Id-\vn \otimes \vn)\mathrm{D}(\eta^{2}(\sh\vn)) \ dx \\
&-\int_{B_{R}(x_{0})}\ep(f'_{j}(x,\mathrm{D}\vn))\sh(Id-\vn\otimes \vn)\mathrm{D}(\eta^{2}(\sh\vn))\ dx \\
&=-\int_{B_{R}(x_{0})}\sh(f'_{j}(x,\mathrm{D}\vn))(Id-\vn \otimes \vn)\eta^{2}\sh(\mathrm{D}\vn) \ dx\\
&-2\int_{B_{R}(x_{0})}\sh(f'_{j}(x,\mathrm{D}\vn))(Id-\vn \otimes \vn)\eta \mathrm{D}\eta (\sh\vn)\ dx \\
&-\int_{B_{R}(x_{0})}\ep(f'_{j}(x,\mathrm{D}\vn))\sh(Id-\vn\otimes \vn)\mathrm{D}(\eta^{2}(\sh\vn))\ dx \\
&=(\mathrm{I})_{1}+(\mathrm{I})_{2}+(\mathrm{I})_{3}.\\
\end{flalign*}

Since $\vn$ takes values into $\mathbb{S}^{N-1}$, $\mathrm{D}\vn=(Id-\vn \otimes \vn)\mathrm{D}\vn$, so
\begin{flalign*}
(\mathrm{I})_{1}&=-\int_{B_{R}(x_{0})}\left(f'_{j}(x+he_{s},\dd v_{j}(x+he_{s}))-f'_{j}(x,\dd v_{j}(x))\right)(\mathrm{Id}-v_{j}\otimes v_{j})\eta^{2}\sh(\dd v_{j}) \ dx\\
&=-\int_{B_{R}(x_{0})}\left(f'_{j}(x,\dd v_{j}(x+he_{s}))-f'_{j}(x,\dd v_{j}(x))\right)(\mathrm{Id}-v_{j}\otimes v_{j})\eta^{2}\sh(\dd v_{j}) \ dx\\
&-\int_{B_{R}(x_{0})}\left(f'_{j}(x+he_{s},\dd v_{j}(x+he_{s}))-f'_{j}(x,\dd v_{j}(x+he_{s}))\right)(\mathrm{Id}-v_{j}\otimes v_{j})\eta^{2}\sh(\dd v_{j}) \ dx\\
&=-\int_{B_{R}(x_{0})}\left(f'_{j}(x,\dd v_{j}(x+he_{s}))-f'_{j}(x,\dd v_{j}(x))\right)\sh(\dd v_{j})\eta^{2}\ dx\\
&+\int_{B_{R}(x_{0})}\left(f'_{j}(x,\dd v_{j}(x+he_{s}))-f'_{j}(x,\dd v_{j}(x))\right)\sh(\mathrm{Id}-v_{j}\otimes v_{j})\ep(\dd v_{j})\eta^{2} \ dx\\
&-\int_{B_{R}(x_{0})}\left(f'_{j}(x+he_{s},\dd v_{j}(x+he_{s}))-f'_{j}(x,\dd v_{j}(x+he_{s}))\right)(\mathrm{Id}-v_{j}\otimes v_{j})\eta^{2}\sh(\dd v_{j}) \ dx\\
&=(\mathrm{I})_{1}^{1}+(\mathrm{I})_{1}^{2}+(\mathrm{I})_{1}^{3}.
\end{flalign*}

Using \eqref{el35} and Lemma \ref{L3}, we estimate $(\mathrm{I})_{1}^{1}$ as
\begin{flalign}\label{e1}
-(\mathrm{I})_{1}^{1}&\le -C^{-1} \int_{B_{R}(x_{0})}\eta^{2}\left( \mu^{2}+\snr{\dd v_{j}}^{2}+\snr{\ep(\dd v_{j})}^{2} \right)^{\frac{p-2}{2}} \snr{\sh (\dd v_{j})}^{2} \ dx\nonumber\\
&\le - c\int_{B_{R}(x_{0})}\snr{\sh((\mu^{2}+\snr{\dd v_{j}}^{2})^{\frac{p-2}{4}}\dd {v_{j}})}^{2} \eta^{2}\ dx.
\end{flalign}

Moreover, by H\"older's inequality, \eqref{el12}, \eqref{90} and Lemma \ref{L6},
\begin{flalign}\label{e2}
\snr{(\mathrm{I})_{1}^{2}}&\le cm^{1/p}\snr{h}^{\alpha}\left[\int_{B_{R}(x_{0})}(1+\snr{\ep(\dd v_{j})}^{q-1}+\snr{\dd v_{j}}^{q-1})\snr{\ep(\dd v_{j} )}\eta^{2} \ dx \ \right]\nonumber\\
&\le cm^{1/p}\snr{h}^{\alpha}\left(\int_{B_{d}(x_{0})}1+\snr{\dd v_{j}}^{q} \ dx \ \right),
\end{flalign}
and, from \eqref{el12}, \eqref{90}, Lemma \ref{L6} and the boundedness of the quantities involved,
\begin{flalign}\label{e3}
\snr{(\mathrm{I})_{1}^{3}}\le c\snr{h}^{\alpha}\int_{B_{R}(x_{0})}(1+\snr{\ep(\dd v_{j})}^{q-1})\snr{\sh (\dd v_{j})}\eta^{2} \ dx\le c\snr{h}^{\alpha}\int_{B_{d}(x_{0})}1+\snr{\dd v_{j}}^{q} \ dx.
\end{flalign}
Moreover, by \eqref{el12}, \eqref{90}, Lemma \ref{L6} and \eqref{el37},
\begin{flalign}\label{e4}
\snr{(\mathrm{I})_{2}}&\le \frac{cm^{1/p}\snr{h}^{\alpha}}{(d-r)}\left[\int_{B_{R}(x_{0})}\snr{f'_{j}(x+he_{s},\dd v_{j}(x+he_{s}))-f'_{j}(x, \dd v_{j}(x+he_{s}))}\eta \ dx\right.\nonumber \\
&\left.\int_{B_{R}(x_{0})}\snr{f'_{j}(x, \dd v_{j}(x+he_{s}))-f'_{j}(x,\dd v_{j}(x))}\eta \ dx \ \right]\nonumber \\
&\le \frac{cm^{1/p}\snr{h}^{\alpha}}{(d-r)}\left[\int_{B_{R}(x_{0})}\left((1+\snr{\ep (\dd v_{j})})^{q-1}+(1+\snr{\dd v_{j}})^{q-1} \right)\eta\ dx\right]\nonumber \\
&\le \frac{cm^{1/p}\snr{h}^{\alpha}}{(d-r)}\left(\int_{B_{d}(x_{0})}1+\snr{\dd v_{j}}^{q} \ dx \ \right),
\end{flalign}
and, in a similar way,
\begin{flalign}\label{e5}
\snr{(\mathrm{I})_{3}}&\le \frac{cm^{1/p}\snr{h}^{\alpha}}{(d-r)}\int_{B_{R}(x_{0})}(1+\snr{\ep (\dd v_{j})}^{q-1})\eta \ dx\nonumber \\
&+cm^{1/p}\snr{h}^{\alpha}\int_{B_{R}(x_{0})}(1+\snr{\ep (\dd v_{j})}^{q-1})\snr{\sh (\dd v_{j})}\eta^{2} \ dx\nonumber \\
&\le \frac{cm^{1/p}\snr{h}^{\alpha}}{(d-r)}\left(\int_{B_{d}(x_{0})}1+\snr{\dd v_{j}}^{q} \ dx\right).
\end{flalign}
Collecting estimates \eqref{e1}-\eqref{e5} we can conclude that
\begin{flalign}\label{e6}
(\mathrm{I})\le - c\int_{B_{R}(x_{0})}\snr{\sh(\mu^{2}+\snr{\dd v_{j}}^{2})^{\frac{p-2}{4}}\dd {v_{j}}}^{2} \ dx+\frac{Cm^{1/p}\snr{h}^{\alpha}}{(d-r)}\left(\int_{B_{d}(x_{0})}1+\snr{\dd v_{j}}^{q} \ dx\ \right),
\end{flalign}
with $c=c(n,N,p,q,C)>0$.\\
\textbf{Term (II).}\\
We use the multiplicative rule for finite differences operators, the fact that
\begin{flalign}
\eme(\sh (v_{j}))&=\sh(v_{j})(x-he_{s})=\hs(v_{j}),\label{e7}\\
\hs(\sh (v_{j}))&=\sh(v_{j})-\hs(v_{j}),\label{e8}
\end{flalign}
and \eqref{el12} to obtain
\begin{flalign}\label{el6}
\snr{(\mathrm{II})}&\le c\int_{B_{R}(x_{0})}(1+\snr{\dd v_{j}}^{q-1})\snr{\hs (\eta^{2})\hs(v_{j})+\eta^{2}(\sh(v_{j})-\hs(v_{j}))}\snr{\dd v_{j}} \ dx\nonumber \\
&\le \frac{cm^{1/p}\snr{h}^{\alpha}}{(d-r)}\left(\int_{B_{d}(x_{0})}1+\snr{\dd v_{j}}^{q} \ dx \ \right),
\end{flalign}
$c=c(n,N,p,q,C)>0$.\\
\textbf{Term (III).}\\
Term (III) contains the same terms of term (II), so, exploiting again identities \eqref{e6}-\eqref{e7} and \eqref{el12} we obtain
\begin{flalign}\label{el8}
\snr{(\mathrm{III})}&\le \frac{cm^{1/p}\snr{h}^{\alpha}}{(d-r)}\left(\int_{B_{d}(x_{0})}1+\snr{\dd v_{j}}^{q} \ dx \ \right),
\end{flalign}
for $c=c(n,N,p,q,C)>0$.\\
\textbf{Term (IV).}\\
Finally proceding as before and taking into account that the $v_{j}$'s and $u$ are bounded, we can conclude that
\begin{flalign}\label{el9}
\snr{(\mathrm{IV})}\le \frac{c\snr{h}^{\alpha}}{(d-r)}\left(\int_{B_{d}(x_{0})}1+\snr{\dd v_{j}}^{q} \ dx\right),
\end{flalign}
$c=c(n,N,p,q,C)>0$.\\
From estimates \eqref{e6}, \eqref{el6}, \eqref{el8}, \eqref{el9}, identity \eqref{EL}, and $\eqref{90}_{3}$, summing on $s \in \{1,\cdots, n\}$ we can conclude that
\begin{flalign}\label{el10}
\int_{B_{r}(x_{0})}\sum_{s=1}^{n}\snr{\sh((\mu^{2}+\snr{\dd v_{j}}^{2})^{\frac{p-2}{4}}\dd {v_{j}})}^{2} \ dx&\le \frac{cm^{1/p}\snr{h}^{\alpha}}{(d-r)}\left(\int_{B_{d}(x_{0})}1+\snr{\dd v_{j}}^{q} \ dx \ \right),
\end{flalign}
with $c=c(n,N,p,q,C)>0$ and $m^{1/p}$ is as in \eqref{el12} and both are independent of $j$.\\
From Lemma \ref{L2} we deduce that
\begin{flalign}\label{e10}
(\mu^{2}+\snr{\mathrm{D}\vn}^{2})^{\frac{p-2}{4}}\mathrm{D}\vn \in L^{\frac{2n}{n-2\tau}}(B_{r}(x_{0})) \quad \mathrm{for\ any\ } \tau \in \left(0,\frac{\alpha}{2}\right).
\end{flalign}
Furthermore, remembering \eqref{h}, there exists $c=c(n,N,p,q,\alpha, \tau, C)<\infty$ such that
\begin{flalign}\label{e17}
\nr{(\mu^{2}+\snr{\mathrm{D}\vn}^{2})^{\frac{p-2}{4}}& \mathrm{D}\vn)}_{L^{\frac{2n}{n-2\tau}}(B_{r}(x_{0}))}\nonumber\\
&\le \frac{c}{(d-r)^{2\tau+\alpha+2}}\left[\int_{B_{d}(x_{0})}\left(1+\snr{\mathrm{D}\vn}^{p}\right) \ dx+\frac{cm^{1/p}}{(d-r)}\int_{B_{d}(x_{0})}1+\snr{\mathrm{D}\vn}^{q} \ dx\right]^{\frac{1}{2}}\nonumber\\
&\le \frac{c(m^{\frac{1}{2p}}+1)}{(d-r)^{2\tau+\alpha+3}}\left(\int_{B_{d}(x_{0})}1+\snr{\mathrm{D}\vn}^{q} \ dx\right)^{ \ \frac{1}{2}}.
\end{flalign}

Setting
\begin{flalign}\label{el28}
a=(2\tau+\alpha+3)\frac{2n}{n-2\tau}, \quad \delta=\frac{p}{q}\frac{n}{n-2\tau}>1, \quad \tilde{m}=(m^{\frac{1}{2p}}+1)^{\frac{2n}{(n-2\tau)}},
\end{flalign}
it follows that
\begin{flalign}\label{e18}
\int_{B_{r}(x_{0})}\snr{\mathrm{D}\vn}^{q\delta} \ dx\le \frac{c\tilde{m}}{(d-r)^{a}}\left(\int_{B_{d}(x_{0})}\left(1+\snr{\mathrm{D}\vn}^{q}\right) \ dx\right)^{\frac{n}{n-2\tau}},
\end{flalign}
where $c=c(n,N,p,q,\alpha, \tau, C)$. From the arbitrariety of $r<d< R$, it follows that
\begin{flalign}\label{e19}
\snr{\mathrm{D}\vn}^{q\delta}\in L^{1}_{\mathrm{loc}}(B_{R}(x_{0})).
\end{flalign}

Now we consider $\gamma>\delta$ to be chosen later. By H\"older inequality we get that
\begin{flalign}\label{e20}
\int_{B_{r}(x_{0})}\snr{\mathrm{D}\vn}^{q\delta} \ dx&\le \frac{c\tilde{m}}{(d-r)^{a}}\left(\int_{B_{d}(x_{0})}\left(1+\snr{\mathrm{D}\vn}\right)^{q\left(1-\frac{\delta}{\gamma}\right)}\left(1+\snr{\mathrm{D}\vn}\right)^{\frac{q\delta}{\gamma}} \ dx\right)^{\frac{\delta q}{p}}\nonumber\\
&\le \frac{c\tilde{m}}{(d-r)^{a}}\left(\int_{B_{d}(x_{0})}\left(1+\snr{\mathrm{D}\vn}^{q\delta}\right) \ dx\right)^{\frac{\delta q}{\gamma p}}\left(\int_{B_{d}(x_{0})}\left(1+\snr{\mathrm{D}\vn}\right)^{q\left(\frac{\gamma-\delta}{\gamma-1}\right)} \ dx\right)^{\frac{\delta q}{p}\frac{\gamma-1}{\gamma}}.
\end{flalign}
Define
\begin{flalign*}
\varepsilon=\frac{\delta}{\gamma}\frac{q}{p}, \quad b=\frac{\gamma-\delta}{\gamma-1}, \quad \lambda=\gamma-1,
\end{flalign*}
so \eqref{e20} becomes
\begin{flalign}\label{e22}
\int_{B_{r}(x_{0})}\snr{\mathrm{D}\vn}^{q\delta} \ dx &\le \left(\int_{B_{d}(x_{0})}\left(1+\snr{\mathrm{D}\vn}^{q\delta}\right) \ dx\right)^{\varepsilon}\nonumber \\
&\times \frac{c\tilde{m}}{(d-r)^{a}}\left(\int_{B_{d}(x_{0})}\left(1+\snr{\mathrm{D}\vn}\right)^{qb} \ dx\right)^{\lambda \varepsilon}.
\end{flalign}
We would like to pick $\gamma$ in such a way that the inequalities
\begin{flalign}\label{e21}
\varepsilon <1, \quad qb\le p, \quad b<\delta
\end{flalign}
are satisfied. Since 
\begin{flalign*}
\delta=\frac{p}{q}\frac{n}{(n-2\tau)}>\frac{p}{q},
\end{flalign*}
the third inequality in \eqref{e21} is implied by the second one. So we only need to check the other two inequalities in \eqref{e21}, which are respectively equivalent to
\begin{flalign}\label{e23}
\delta \frac{q}{p}<\gamma, \quad \gamma\le \frac{\delta q-p}{q-p}.
\end{flalign}
All in all, we can find $\gamma$ with the aforementioned features if and only if
\begin{flalign*}
\delta\frac{q}{p}<\frac{\delta q-p}{q-p},
\end{flalign*}
which is equivalent to
\begin{flalign*}
\frac{q}{p}<\frac{n+\alpha}{n}.
\end{flalign*}
This last condition is matched, by \eqref{I2}.\\
We turn back to \eqref{e22} and we apply Young inequality to obtain
\begin{flalign}\label{e24}
\int_{B_{r}(x_{0})}\snr{\mathrm{D}\vn}^{q\delta} \ dx &\le \frac{c\tilde{m}^{\frac{1}{(1-\varepsilon)}}}{(d-r)^{\frac{a}{1-\varepsilon}}}\left(\int_{B_{d}(x_{0})}(1+\snr{\mathrm{D}\vn})^{p} \ dx\right)^{\frac{\lambda \varepsilon}{1-\varepsilon}}+\frac{1}{2}\int_{B_{d}(x_{0})}(1+\snr{\mathrm{D}\vn})^{q\delta} \ dx,
\end{flalign}
where $c=c(n,N,p,q,\alpha, \tau, C)>0$.

Notice that the right-hand side of \eqref{e24} is finite because of \eqref{e18}. Now we apply Lemma \ref{L4} with the choice
\begin{flalign*}
R_{0}=R-\nu, \quad 0<\nu<R-\rho,
\end{flalign*}
\begin{flalign*}
A=\tilde{m}^{\frac{1}{1-\varepsilon}}\left(\int_{B_{R}(x_{0})}1+\snr{\mathrm{D}\vn}^{p} \ dx\right)^{\frac{\lambda\varepsilon}{1-\varepsilon}}, \quad h(r)=\int_{B_{r}(x_{0})}1+\snr{\mathrm{D}\vn}^{q\delta} \ dx,
\end{flalign*}
therefore we use \eqref{e24} for the values $\rho \le r<d \le R-\nu$. This choice is possible because of \eqref{e19}, so $h(r)$ is bounded over $[\rho, R-\nu]$. It follows that
\begin{flalign}\label{e25}
\int_{B_{\rho}(x_{0})}\snr{\mathrm{D}\vn}^{\frac{np}{n-2\tau}} \ dx &\le \frac{c\tilde{m}^{\frac{1}{(1-\varepsilon)}}}{(R-\nu-\rho)^{\frac{a}{1-\varepsilon}}}\left(\int_{B_{R}(x_{0})}(1+\snr{\mathrm{D}\vn})^{p} \ dx\right)^{\frac{\lambda \varepsilon}{1-\varepsilon}},
\end{flalign}
with $c=c(n,N,p,q, \alpha, \tau, C)$ and $\tilde{m}$ is as in \eqref{el29}. Since both $c$ and $\tilde{m}$ do not depend on $\nu$, we let $\nu \to 0$, so \eqref{e25} is still true for $\sigma=0$. 
Now define
\begin{flalign}\label{el29}
\tilde{a}=\frac{a}{1-\varepsilon}, \quad \beta=\frac{\lambda\varepsilon}{1-\varepsilon}, \quad  \tilde{\beta}=\beta+\frac{2n}{(n-2\tau)(1-\varepsilon)}.
\end{flalign}
Fixing $\tau$ in such a way that $t<\frac{np}{n-2\tau}$, from \eqref{e25} and \eqref{el32} we get
\begin{flalign*}
\int_{B_{\rho}(x_{0})}\snr{\mathrm{D}v_{j}}^{t} \ dx\le \frac{c\tilde{m}^{\frac{1}{(1-\varepsilon)}}}{(R-\rho)^{\tilde{\alpha}}}\left(\int_{B_{R}(x_{0})}f(x,\mathrm{D}u_{j})+\frac{1}{p}\snr{v_{j}-u}^{p}+1\ dx+\varepsilon_{j}\int_{B_{R}(x_{0})}(1+\snr{\dd v_{j}}^{2})^{q/2} \ dx\right)^{ \ \beta},
\end{flalign*}
where $c$ has the aforementioned dependencies.\\
As already pointed out in \eqref{el26}, the sequence $(v_{j})_{j \in \N}$ is uniformly bounded in $W^{1,p}(B_{R}(x_{0}), \SN)$, so, there exists $w \in W^{1,p}_{u}(B_{R}(x_{0}),\SN)$ such that (up to extract a subsequence) $\dd v_{j}\rightharpoonup_{j \nearrow \infty} \dd w$ weakly in $L^{p}(B_{R}(x_{0}),\mathbb{R}^{N\times n})$. By weak lower semicontinuity, we get
\begin{flalign}\label{e27}
\int_{B_{\rho}(x_{0})}\snr{\mathrm{D}w}^{t} \ dx&\le \liminf_{j \nearrow \infty}\int_{B_{\rho}(x_{0})}\snr{\mathrm{D}\vn}^{t} \ dx\nonumber\\
&\le \liminf_{j \nearrow \infty}\frac{c\tilde{m}^{\frac{1}{(1-\varepsilon)}}}{(R-\rho)^{\tilde{\alpha}}}\left(\int_{B_{R}(x_{0})}\left(f(x,\mathrm{D}\vn)+1\right)+\frac{1}{p}\snr{\vn-u}^{p} \ dx+\varepsilon_{j}\int_{B_{R}(x_{0})}(1+\snr{\mathrm{D}\vn}^{2})^{\frac{q}{2}}\ dx\right)^{ \ \beta}\nonumber \\
&\le \liminf_{j \nearrow \infty}\frac{c\tilde{m}^{\frac{1}{(1-\varepsilon)}}}{(R-\rho)^{\tilde{\alpha}}}\left(\int_{B_{R}(x_{0})}\left(f(x,\mathrm{D}\bar{v}_{j})+1\right)+\frac{1}{p}\snr{\bar{v}_{j}-u}^{p} \ dx+\varepsilon_{j}\int_{B_{R}(x_{0})}(1+\snr{\mathrm{D}\bar{v}_{j}}^{2})^{\frac{q}{2}}\ dx\right)^{ \ \beta}\nonumber \\
&=\frac{c}{(R-\rho)^{\tilde{\alpha}}}\left(\int_{B_{R}(x_{0})}f(x,\mathrm{D}u)\ dx+1\right)^{\ \tilde{\beta}},
\end{flalign}
where $c=c(n,N,p,q,\tau,\alpha), \tilde{\alpha}$, $\beta$ and $\tilde{\beta}$ are as in \eqref{el29} and we used the explicit expression of $\tilde{m}$. Moreover
\begin{flalign}\label{e28}
\int_{B_{R}(x_{0})}f(x,\mathrm{D}w) \ dx &\le \int_{B_{R}(x_{0})}f(x,\mathrm{D}w)+\frac{1}{p}\snr{w-u}^{p} \ dx \nonumber \\ 
&\le \liminf_{j \nearrow \infty} \int_{B_{R}(x_{0})}f(x,\mathrm{D}\vn)+\frac{1}{p}\snr{\vn-u}^{p} \ dx+\varepsilon_{j}\int_{B_{R}(x_{0})}(1+\snr{\mathrm{D}\vn}^{2})^{\frac{q}{2}}\ dx\nonumber \\
&\le \liminf_{j \nearrow \infty}\int_{B_{R}(x_{0})}f(x,\mathrm{D}\bar{v}_{j})+\frac{1}{p}\snr{\bar{v}_{j}-u}^{p} \ dx+\varepsilon_{j}\int_{B_{R}(x_{0})}(1+\snr{\mathrm{D}\bar{v}_{j}}^{2})^{\frac{q}{2}}\ dx\nonumber\\
&\le \liminf_{j \nearrow \infty}\int_{B_{R}(x_{0})}f(x,\mathrm{D}\bar{v}_{j}) \ dx +o(j) =\int_{B_{R}(x_{0})}f(x,\mathrm{D}u) \ dx.
\end{flalign}
We stress that, in particular, $w \in W^{1,p}_{u}(B_{R}(x_{0}),\mathbb{S}^{N-1})$, thus the minimality of $u$ gives
\begin{flalign}\label{e29}
\int_{B_{R}(x_{0})}f(x,\mathrm{D}w) \ dx \ge \int_{B_{R}(x_{0})}f(x,\mathrm{D}u) \ dx.
\end{flalign}
From \eqref{e28} and \eqref{e29} we get
\begin{flalign*}
\int_{B_{R}(x_{0})}f(x,\mathrm{D}u) \ dx \le \int_{B_{R}(x_{0})}f(x,\mathrm{D}w) \ dx \le \int_{B_{R}(x_{0})}f(x,\mathrm{D}w)+\frac{1}{p}\snr{w-u}^{p} \ dx \le \int_{B_{R}(x_{0})}f(x,\mathrm{D}u) \ dx,
\end{flalign*}
so
\begin{flalign*}
\int_{B_{R}(x_{0})}f(x,\mathrm{D}u) \ dx =\int_{B_{R}(x_{0})}f(x,\mathrm{D}w) \ dx \quad \mathrm{and} \quad \int_{B_{R}(x_{0})}\snr{w-u}^{p} \ dx =0,
\end{flalign*}
and the higher integrability of $\mathrm{D}u$ is proved.
\end{proof}

\begin{remark}
\emph{The hypotheses considered in the previous Theorem can be sometimes restrictive, and, in particular, do not cover the model energies ($\mathrm{E}_{1}$), ($\mathrm{E}_{3}$) and ($\mathrm{E}_{4}$). However we can fix this issue by replacing (H3) by any of the following
\begin{flalign*}
&(\mathrm{H3.1}) \quad C^{-1}(\mu^{2}+\snr{\zeta_{1}}^{2}+\snr{\zeta_{2}}^{2})^{\frac{p(x)-2}{2}}\snr{\zeta_{1}-\zeta_{2}}^{2}\le (f'(x,\zeta_{1})-f'(x,\zeta_{2}),\zeta_{1}-\zeta_{2}),\\
&(\mathrm{H3.2}) \quad C^{-1}\sum_{i=1}^{n}(\mu_{i}^{2}+\snr{\zeta_{1,i}}^{2}+\snr{\zeta_{2,i}}^{2})^{\frac{p_{i}-2}{2}}\snr{\zeta_{1,i}-\zeta_{2,i}}^{2}\le (f'(x,\zeta_{1})-f'(x,\zeta_{2}),\zeta_{1}-\zeta_{2}),
\end{flalign*}
where $\zeta_{1}$, $\zeta_{2}\in \mathbb{R}^{N\times n}$, $x \in \Omega$, $\mu, \mu_{i}\in [0,1]$ and $\zeta_{1,i}=(\zeta_{1,i}^{k})_{k \in \{1,\cdots, N\}}$, $\zeta_{2,i}=(\zeta_{2,i}^{k})_{k \in \{1,\cdots, N\}}\in \mathbb{R}^{N}$.\\
Since (H3.1) and (H3.2) score only the coercivity condition, we just need to modify \eqref{e1} as done in \cite{b15}, Theorem 5, to conclude that the result in Theorem \ref{T1} remains true also for ($\mathrm{E}_{1}$), ($\mathrm{E}_{3}$) and ($\mathrm{E}_{4}$).
}
\end{remark}

\begin{corollary}\label{C2}
Let $u\in W^{1,p}_{\mathrm{loc}}(\Omega, \SN)$ be a local minimizer of $\mathcal{F}$, with $\mathcal{F}$ belonging to classes ($\mathcal{A}$), ($\mathcal{B}$) or ($\mathcal{C}$). Then $u \in W^{1,q}_{\mathrm{loc}}(\Omega,\mathbb{S}^{N-1})$.
\end{corollary}
\begin{proof}
Theorem \ref{T1} yields that $\mathrm{D}u \in L^{t}_{\mathrm{loc}}(\Omega,\mathbb{R}^{N\times n})$ for any $t\in \left(1,\frac{np}{n-\alpha}\right)$. By assumption, we know that $\frac{q}{p}<\frac{n+\alpha}{n}$. A direct computation shows that $\frac{np}{n-\alpha}>p\frac{n+\alpha}{n}$, so, in particular $\mathrm{D}u \in L^{q}_{\mathrm{loc}}(\Omega,\mathbb{R}^{N\times n})$.
\end{proof}

\begin{remark}\label{L69}
\emph{For $\mathbb{S}^{1}$-valued maps, the whole analysis can be considerably simplified in presence of a radial structure. \\ Precisely,} \emph{let $\Omega \subset \mathbb{R}^{n}$ be simply connected and }
\begin{flalign*}
F\colon \Omega \times \mathbb{R}^{2\times n}\to \mathbb{R}\cup \{\infty \}, \quad F(x,\zeta)=f(x,\snr{\zeta}^{2})
\end{flalign*}
\emph{is a normal, convex integrand satisfying, for some exponent} $p \in [2,\infty)$, \emph{the coercivity condition}
\begin{flalign*}
F(x,\zeta ) \ge \snr{\zeta}^{p}
\end{flalign*}
\emph{for all }$\zeta \in \mathbb{R}^{2\times n}$ \emph{and a. e.} $x \in \Omega$. \\
\emph{We consider the corresponding variational integral}
\begin{flalign}\label{vp69}
\mathcal{F}(u,\Omega)=\int_{\Omega}f(x,\snr{\dd w}^{2}) \ dx,
\end{flalign}
\emph{with $w\in W^{1,p}(\Omega, \mathbb{S}^{1})$. Fix $g\in W^{1,p}(\Omega, \mathbb{S}^{1})$ so that $F(\cdot , \dd g(\cdot))\in L^{1}(\Omega)$. Using the direct method it is easy to prove the existence of a minimizer of \eqref{vp69},  $\bar{u}\in W^{1,p}_{g}(\Omega, \mathbb{S}^{1})$.\\
Notice that, as well known from the Lifting Theorem, \cite{b4, b5, b34}, we can find $\bar{\varphi}\in W^{1,p}(\Omega)$ so that $\bar{u}=e^{i\bar{\varphi}}$ (with the obvious identification $\mathbb{C}\cong \mathbb{R}^{2}$) and from the trace theorem it follows that $\left.g\right |_{\partial \Omega}=\left.e^{i\bar{\varphi}}\right |_{\partial \Omega}$. Now, if $\varphi \in W^{1,p}_{\bar{\varphi}}(\Omega)$, then clearly $e^{i\varphi}\in W^{1,p}_{g}(\Omega, \mathbb{S}^{1})$, and since
\begin{flalign*}
\snr{\dd (e^{i\varphi})}=\snr{\dd \varphi},
\end{flalign*}
we deduce that $f(x,\snr{\dd \bar{\varphi}}^{2})=F(x,\dd\bar{u})\in L^{1}(\Omega)$ and 
\begin{flalign*}
\int_{\Omega}f(x,\snr{\dd \bar{\varphi}}^{2}) \ dx\le \int_{\Omega}f(x,\snr{\dd \varphi}^{2}) \ dx
\end{flalign*}
for all $\varphi \in W^{1,p}_{\bar{\varphi}}(\Omega)$. Thus the variational problem for \eqref{vp69} reduces to a variational problem for real-valued functions, and we can therefore use the much stronger results available in the scalar, radial case.\\
}
\emph{Such an approach is very interesting, since it allows tracing a vectorial variational problem back to a scalar one, and, by common sense, scalar minimizers have more hopes for higher regularity than their vectorial "counterpart".\\
For instance, we may apply these ideas on the model energy for Double Phase Variational Integral
\begin{flalign*}
(\mathrm{E}_{2}) \quad \mathcal{F}(u,\Omega)=\int_{\Omega}\snr{\dd u}^{p}+a(x)\snr{\dd u}^{q} \ dx.
\end{flalign*}
Let us recall that their regularity theory is now very well-understood in the unconstrained case, see \cite{b50, b51, b10}. In particular, by Theorem 1.4 in \cite{b10} we obtain that $u \in C^{1,\beta}_{\mathrm{loc}}(\Omega, \mathbb{S}^{1})$.\\
We conclude by stressing that, if $h\colon \Omega\times \mathbb{R}^{2\times n}\to \mathbb{R}$ is a Carath\'eodory integrand such that
\begin{flalign*}
c^{-1}f(x,\snr{\zeta}^{2})\le h(x,\zeta)\le cf(x,\snr{\zeta}^{2}), \quad c>1,
\end{flalign*}
and $u=e^{i\bar{\varphi}}$, is any minimizer of the functional
\begin{flalign*}
\mathcal{H}(w, \Omega)=\int_{\Omega}h(x,\dd w(x)) \ dx, \quad w\colon \Omega \to \mathbb{S}^{1},
\end{flalign*}
then $\bar{\varphi}$ is an unconstrained quasi-minimizer of $\mathcal{F}(w,\Omega)=\int_{\Omega}f(x,\snr{\dd w}^{2}) \ dx$. In fact, as above,
\begin{flalign*}
\int_{\Omega}f(x, \snr{\dd \bar{\varphi}}^{2}) \ dx&=\int_{\Omega}f(x, \snr{\dd \bar{u}}^{2}) \ dx\\
&\le c\int_{\Omega}h(x, \dd \bar{u}) \ dx\\
&\le c\int_{\Omega}h(x, \dd w) \ dx\le c^{2}\int_{\Omega}f(x, \snr{\dd \varphi}^{2}) \ dx,
\end{flalign*}
for all $\varphi \in W^{1,p}_{\bar{\varphi}}(\Omega)$.\\
A detailed analysis on this matter will be the object of a forthcoming paper, \cite{bluff}.
}
\end{remark}

\textbf{Acknowledgments.} The author wishes to thank Prof. J. Kristensen for the many discussions and suggestions during the completion of this paper. Further thanks are extended to the anonymous reviewers, for their careful reading of the manuscript and their useful comments.



    
    
    
    

\end{document}